\documentclass[]{article}

\usepackage{dsfont}
\usepackage{graphicx}
\usepackage{amsmath}
\usepackage{amssymb}
\usepackage{amsthm}
\usepackage{mathrsfs}
\usepackage{pxfonts}
\usepackage{enumerate}
\usepackage{color}
\usepackage{mathdots}
\usepackage{sectsty}
\usepackage{tikz}
\usepackage{adjustbox}
\usepackage{caption}
\usepackage{bbold}
\usepackage[hidelinks]{hyperref}
\allowdisplaybreaks
\usepackage[nameinlink,capitalise,noabbrev]{cleveref}

\usepackage[shortlabels]{enumitem}

\sectionfont{\scshape\centering\fontsize{11}{14}\selectfont}
\subsectionfont{\scshape\fontsize{11}{14}\selectfont}
\usepackage{fancyhdr}
\usepackage[nottoc,notlot,notlof]{tocbibind}

\newcommand\shorttitle{Collision and non-collision on configuration space}
\newcommand\authors{T. Assiotis and K. Suzuki}

\numberwithin{equation}{section}
\newtheorem{thm}{Theorem}[section]
\newtheorem{lemma}[thm]{Lemma}

\newtheorem{lem}[thm]{Lemma}

\newtheorem{remark}[thm]{Remark}
\newtheorem{prop}[thm]{Proposition}

\newtheorem{assumption}[thm]{Assumption}

\newtheorem*{theorem*}{Theorem}

\newcommand{\supp}{\mathrm{supp}}
\newcommand{\define}{\overset{\mathrm{def}}{=}}

\newcommand{\R}{\mathbb{R}}
\newcommand{\N}{\mathbb{N}}
\newcommand{\eps}{\varepsilon}
\newcommand{\Conf}{\boldsymbol{\Theta}}
\newcommand{\cE}{\mathcal{E}}
\newcommand{\cD}{\mathcal{D}}
\renewcommand{\Cap}{\mathrm{Cap}}
\newcommand{\Ir}[1]{I_{#1}}

\newcommand{\Sine}{\mathsf{Sine}}

\usepackage{xparse}

\ExplSyntaxOn
\NewDocumentCommand{\makeabbrev}{mmm}
 {
  \yoruk_makeabbrev:nnn { #1 } { #2 } { #3 }
 }

\cs_new_protected:Npn \yoruk_makeabbrev:nnn #1 #2 #3
 {
  \clist_map_inline:nn { #3 }
   {
    \cs_new_protected:cpn { #2 } { #1 { ##1 } }
   }
 }
 \ExplSyntaxOff
 
\makeabbrev{\textbf}{tbf#1}{a,b,c,d,e,f,g,h,i,j,k,l,m,n,o,p,q,r,s,t,u,v,w,x,y,z,A,B,C,D,E,F,G,H,I,J,K,L,M,N,O,P,Q,R,S,T,U,V,W,X,Y,Z}

\makeabbrev{\textbf}{bf#1}{a,b,c,d,e,f,g,h,i,j,k,l,m,n,o,p,q,r,s,t,u,v,w,x,y,z,A,B,C,D,E,F,G,H,I,J,K,L,M,N,O,P,Q,R,S,T,U,V,W,X,Y,Z}

\makeabbrev{\textsf}{tsf#1}{a,b,c,d,e,f,g,h,i,j,k,l,m,n,o,p,q,r,s,t,u,v,w,x,y,z,A,B,C,D,E,F,G,H,I,J,K,L,M,N,O,P,Q,R,S,T,U,V,W,X,Y,Z}

\makeabbrev{\mathsf}{mss#1}{a,b,c,d,e,f,g,h,i,j,k,l,m,n,o,p,q,r,s,t,u,v,w,x,y,z,A,B,C,D,E,F,G,H,I,J,K,L,M,N,O,P,Q,R,S,T,U,V,W,X,Y,Z}

\makeabbrev{\mathfrak}{mf#1}{a,b,c,d,e,f,g,h,i,j,k,l,m,n,o,p,q,r,s,t,u,v,w,x,y,z,A,B,C,D,E,F,G,H,I,J,K,L,M,N,O,P,Q,R,S,T,U,V,W,X,Y,Z}

\makeabbrev{\mathrm}{mrm#1}{a,b,c,d,e,f,g,h,i,j,k,l,m,n,o,p,q,r,s,t,u,v,w,x,y,z,A,B,C,D,E,F,G,H,I,J,K,L,M,N,O,P,Q,R,S,T,U,V,W,X,Y,Z}

\makeabbrev{\mathbf}{mbf#1}{a,b,c,d,e,f,g,h,i,j,k,l,m,n,o,p,q,r,s,t,u,v,w,x,y,z,A,B,C,D,E,F,G,H,I,J,K,L,M,N,O,P,Q,R,S,T,U,V,W,X,Y,Z}

\makeabbrev{\mathcal}{mc#1}{A,B,C,D,E,F,G,H,I,J,K,L,M,N,O,P,Q,R,S,T,U,V,W,X,Y,Z}

\makeabbrev{\mathbb}{mbb#1}{A,B,C,D,E,F,G,H,I,J,K,L,M,N,O,P,Q,R,S,T,U,V,W,X,Y,Z}

\makeabbrev{\mathscr}{ms#1}{A,B,C,D,E,F,G,H,I,J,K,L,M,N,O,P,Q,R,S,T,U,V,W,X,Y,Z}

\title{\large \bf COLLISION AND NON-COLLISION FOR DIFFUSIONS ON CONFIGURATION SPACE}
\author{\small THEODOROS ASSIOTIS AND KOHEI SUZUKI}
\date{}

\begin{document}

\maketitle

\begin{abstract}
We develop criteria for collision and non-collision of reversible infinitely many  interacting diffusion processes in the real line. The approach is potential-theoretic and is based on capacity estimates for symmetric Dirichlet forms on the configuration space.  Our main results are model-independent in the sense that no determinantal or Pfaffian structure, prescribed interaction potential, or explicit labelled stochastic differential equation is required. The non-collision criterion involves only the second and third correlation functions of the reversible measure, whereas the collision criterion is based on a local lower bound for a finite-volume conditional density of the reversible measure. As an application, we identify the sharp collision threshold for the diffusion associated with the $\Sine_\beta$-symmetric Dirichlet form: the collision set is polar if and only if $\beta\ge 1$. This provides an infinite-particle counterpart of the classical collision threshold for the finite-particle Dyson Brownian motion with inverse temperature~$\beta$.
\end{abstract}
\tableofcontents

\section{Introduction}

Interacting diffusions describe the random motion of many particles whose trajectories are not independent, but influence one another. Such systems arise naturally, for instance,  in statistical mechanics, stochastic analysis, and random matrix theory. In one spatial dimension, a basic qualitative question is whether two particles can occupy the same position. This question is simple to state, but it captures an important structural feature of the dynamics.

In one dimension, non-collision preserves the order of particles. This gives a canonical way to follow individual particles through time and allows an unlabeled configuration-valued process to be regarded as an ordered (or labelled) particle system. If collisions can occur, the diagonal is no longer an inaccessible boundary: particles may cross, reflect, coalesce, or lose a canonical labeling. Thus the meaning of the dynamics at and after collision time depends on additional choices of boundary conditions. In this sense, the collision problem is not merely a pathwise detail, but determines the qualitative interpretation of the process.

The problem becomes especially delicate for infinite particle systems:
each particle is surrounded by infinitely many other particles, in particular when interactions are of long-range, the collision and non-collision problem becomes more challenging as it involves interactions from infinitely many other particles. The local behavior near a possible collision must therefore be studied together with the surrounding random environment.

\smallskip
The purpose of this note is to  determine whether infinite-interacting particles of one-dimensional Brownian motions collide. Our approach is model-independent in the sense that we do not impose interactions to have a particular form: we rather study  $\mu$-reversible diffusion processes on the configuration space \(\Conf\) for general probability measure $\mu$ on~$\Conf$:
\[
\Conf \define \left\{\theta=\sum_i\delta_{x_i}:\theta(K)<\infty\text{ for every compact }K\subset\mathbb R\right\}.
\]
Since the occurrence of a collision is independent of how the particles are labelled, it is natural to formulate the problem on the configuration space rather than on an infinite product space of labelled particles.
The collision set is then given by
\[
\mathscr C \define \{\theta\in\Conf:\theta(\{x\})\ge2\text{ for some }x\in\mathbb R\}.
\]

In the configuration-space formulation, a collision is exactly the event that the diffusion reaches the collision set \(\mathscr C\). The collision problem is therefore a hitting problem for a distinguished subset of the state space. We tackle this problem from the Dirichlet form perspective. 
A Dirichlet form provides an analytic description of a diffusion process: it measures the infinitesimal energy of functions on the state space, and, under suitable closedness and topological regularity assumptions, it determines the corresponding Markov process. In the configuration-space setting, this allows us to study an interacting particle system through an energy form on functions of configurations, rather than through labelled stochastic differential equations.  The Dirichlet form associated with the diffusion encodes this probabilistic question in a {\it potential analytic} form through capacity: sets of zero capacity are invisible to the diffusion, while sets of positive capacity may be hit. Thus collision and non-collision can be studied by estimating the capacity of the collision set~\(\mathscr C\).

\smallskip
Our first result gives a non-collision criterion in terms of the local repulsion of the reversible measure~$\mu$ under a mild assumption. We assume that the second and third factorial correlation functions of $\mu$ exist and satisfy suitable upper bounds near the diagonal. Roughly speaking, if the two-point correlation function vanishes at least linearly as two points merge, then the collision set has zero capacity. The third correlation function is used only to control cross terms in the energy estimate. This criterion applies to many point processes, in particular, to the \(\mathsf{Sine}_\beta\) process for \(\beta\ge1\), using the recent correlation estimates of the first author and Najnudel~\cite{AssiotisNajnudel}.

Our second result gives a complementary collision criterion. It shows that if, in a finite window and on a positive-measure set of exterior configurations, the conditional density near a two-particle collision is bounded below by
\[
|x-y|^\beta
\]
with \(0\le\beta<1\), then the collision set has positive capacity. This is the opposite regime from the non-collision theorem: the repulsion near the diagonal is too weak to make the diagonal invisible to the associated diffusion processes. We can apply it to the case $\mu=\Sine_\beta$ with $0 \le \beta <1$ by leveraging the Dobrushin--Lanford--Ruelle (DLR) equation for $\mathsf{Sine}_\beta$ established in~\cite{DLR}.

Combining these two criteria yields a capacity dichotomy for the diffusion process associated with the Dirichlet form with reversible measure~\(\mathsf{Sine}_\beta\):
\[
\Cap(\mathscr C)=0 \quad\text{for} \quad \beta\ge1,
\qquad
\Cap(\mathscr C)>0 \quad\text{for} \quad 0<\beta<1.
\]
Thus the value \(\beta=1\) marks the boundary between non-collision and collision from the viewpoint of capacity. The threshold \(\beta=1\) is the same as in finite-dimensional
\(\beta\)-Dyson Brownian motion (C\'epa--L\'epingle~\cite{CepaLep97}):
\begin{equation}\label{Dyson}
\mathrm{d}\boldsymbol{\lambda}_i(t)=\mathrm{d}\mathsf{w}_i(t)+\frac{\beta}{2} \sum_{j\neq i}^N \frac{1}{\boldsymbol{\lambda}_i(t)-\boldsymbol{\lambda}_j(t)}\mathrm{d}t, \ \ i=1,\dots, N,  
\end{equation}
where the $\mathsf{w}_i$, $i=1,\dots, N$ are independent standard Brownian motions. 
The particles are
non-colliding for \(\beta\ge1\), whereas collisions may occur for \(0<\beta<1\).
Our result with a particular application to \(\mathsf{Sine}_\beta\) may therefore be regarded as an infinite-dimensional 
counterpart of this classical finite-dimensional collision threshold for Dyson Brownian motion, see \eqref{BulkDyson} below.

\paragraph{Literature comparison}
This work may be viewed as a model independent development of Osada's
capacity approach to collision and non-collision. In \cite{Osa04}, Osada proved
that the collision set has zero capacity for determinantal random point fields
whose correlation kernel is locally Lipschitz; this includes the sine-kernel
process, and hence the case~\(\mathsf{Sine}_2\). He also exhibited collision
examples within a particular class of determinantal random point fields with
less regular, H\"older-continuous kernels, and obtained a non-collision result
for certain Gibbs measures under assumptions on a prescribed interaction
potential. Thus, although \cite{Osa04} contains both non-collision results and
collision examples, its arguments are tied either to an explicit determinantal
representation of the equilibrium measure or to a specified Gibbsian model.

In this note,  the level of generality at which the collision problem is formulated is considerably broader.
We impose neither an algebraic structure on the equilibrium measure, such as a
determinantal or Pfaffian representation, nor a model-specific description in
terms of a Hamiltonian, pair interaction potential, logarithmic derivative, or
labelled infinite-dimensional stochastic differential equation. For the 
non-collision criterion, our assumptions involve only local upper bounds for the
second and third factorial correlation functions. No information about
correlation functions of order four or higher is required. This is in contrast
with the sectorwise mollification argument in \cite{Osa04}, where the
determinantal structure provides explicit control of the full hierarchy of
finite-volume densities. For the collision criterion, we require only a local lower
bound for one finite-volume conditional sector density, on a set of exterior
configurations of positive measure. In particular, the remaining particles and
the random exterior environment do not need to be described explicitly, and the
lower bound is not required uniformly over all exterior configurations.



\subsection{Setting}  
\paragraph{Configuration space terminology}
 For any $r>0$ we set
\[
\Ir{r}\define(-r,r).
\]
For $n\ge 0$ define the $n$-particle sector on $\Ir{r}$:
\[
\Conf_{r}^{n}\define\{\theta\in\Conf:\ \theta(\Ir{r})=n\}.
\]
Then $\Conf=\bigsqcup_{n=0}^{\infty}\Conf_{r}^{n}$. For $A \subset \R$, let $\pi_A: \Theta \to \Theta(A)$ be the restriction~$\theta \mapsto \theta|_{A}$ for the configuration~$\theta$ onto~$A$. We simply write $\pi_r$ when $A=I_r$.
Let $I_r^k=I_r \times I_r \times \cdots \times I_r$ be the $k$-product of $I_r$. 
If $F$ is $\sigma[\pi_r]$-measurable, where $\sigma$ denotes the corresponding sigma algebra, then it has a unique symmetric representative
$F_r^n:\Ir{r}^n\to\R$ satisfying $F(\theta)=F_r^n(x_1,\dots,x_n)$ whenever
$\theta|_{\Ir{r}}=\sum_{i=1}^n\delta_{x_i}$.
More generally, for $s>0$ and $\eta\in \Conf(\Ir{s}^c)$, an $\Ir{s}$-representation of a function $F$ is a family of symmetric functions
\[
F_{s,\eta}^n:\Ir{s}^n\to\R\qquad (n\ge 0)
\]
such that
\begin{align} \label{e:SR}
F\!\left(\eta+\sum_{i=1}^n\delta_{x_i}\right)=F_{s,\eta}^n(x_1,\dots,x_n)
\qquad\text{for }(x_1,\dots,x_n)\in \Ir{s}^n.
\end{align}
When $F$ is $\sigma[\pi_s]$-measurable, this representation is independent of $\eta$, and we  write $F_s^n$.

\paragraph{Test functions} For $r>0$, let $\mathcal I_r$ be the class of bounded $\sigma[\pi_r]$-measurable functions and set
\[
\mathcal B_\circ\define\bigcup_{r>0}\mathcal I_r.
\]
Following \cite{Osa96}, for each $r>0$ we define
\[
\mathcal D^{\mathrm{loc}}_{\infty, r}
\define
\Bigl\{F\in\mathcal I_r:\ F^n_{r}\in C^\infty(\Ir{r}^n) \  \text{for every~$n$}\ \Bigr\},
\]
and
\[
\mathcal D^{\mathrm{loc}}_{\infty}
\define
\Bigl\{F\in\mathcal B_\circ:\ \text{for every } s>0,\ \eta\in\Conf(\Ir{s}^c),\ \text{and } n\ge0,\ F^n_{s,\eta}\in C^\infty(\Ir{s}^n)\Bigr\}.
\]
Thus every $F\in\mathcal D^{\mathrm{loc}}_\infty$ is local in the sense that $F\in\mathcal I_r$ for some $r$, and smooth in the sense that all finite-volume representatives are smooth. By the argument of \cite{Osa96}, $\mathcal D^{\mathrm{loc}}_\infty\subset C(\Conf)$ the space of continuous functions with respect to vague topology on $\Conf$, i.e., the topology induced by the duality with compactly supported continuous functions in~$\R$. By chain rule,  if $F_1,\dots,F_k\in\mathcal D^{\mathrm{loc}}_\infty$ and $\Phi\in C_b^\infty(\R^k)$, then $\Phi(F_1,\dots,F_k)\in\mathcal D^{\mathrm{loc}}_\infty$.
Moreover, if $\varphi_1,\dots,\varphi_k\in C_c^\infty(\R)$ and $\Phi\in C_b^\infty(\R^k)$, then the cylinder functions
\[
\theta\mapsto \Phi\bigl(\langle \varphi_1,\theta\rangle,\dots,\langle \varphi_k,\theta\rangle\bigr)
\]
belong to $\mathcal D^{\mathrm{loc}}_\infty$, where $\langle \varphi, \theta \rangle \define \theta (\varphi)=\int_\R \varphi(x) \theta(dx)$ for $\varphi \in C_c^\infty(\R)$. 
We now define the capacity of the form. 

\paragraph{Dirichlet form} We define the square field operator on functions which are measurable with respect to one fixed window. If $F,G\in\mathcal D^{\mathrm{loc}}_{\infty,r}$ and $\theta\in\Conf_{r}^{n}$ has $\Ir{r}$-coordinate $x=(x_1,\dots,x_n)$, set
\[
\mathbb D_r[F,G](\theta)
\define
\frac12\sum_{i=1}^n \frac{\partial F_r^n}{\partial x_i}(x)\,\frac{\partial G_r^n}{\partial x_i}(x),
\qquad
\mathbb D_r[F]\define\mathbb D_r[F,F].
\]
Since $\sigma[\pi_r]$-measurable functions do not depend on configurations outside $I_r$, the following equality is immediate by the definition of $\mathbb D_r$ above:  
let $0<r<s$ and  $F,G\in \mathcal D^{\mathrm{loc}}_{\infty,r}\cap \mathcal D^{\mathrm{loc}}_{\infty,s}$. Then
\begin{align} \label{lem:square-field-compat}
\mathbb D_r[F,G](\theta)=\mathbb D_s[F,G](\theta)
\qquad\text{for every }\theta\in\Conf.
\end{align}
Consequently, if $F,G\in\mathcal D^{\mathrm{loc}}_\infty$ and $t>0$ is such that $F,G\in\mathcal D^{\mathrm{loc}}_{\infty,t}$, then the quantity $\mathbb D_t[F,G]$ is independent of the particular choice of such $t$.
For $F,G\in\mathcal D^{\mathrm{loc}}_\infty$, choose $t>0$ such that $F,G\in\mathcal D^{\mathrm{loc}}_{\infty,t}$ and define
\[
\mathbb D[F,G]\define \mathbb D_t[F,G],
\qquad
\mathbb D[F]\define\mathbb D[F,F].
\]
By~\eqref{lem:square-field-compat}, this is well-defined.
For a Borel (with respect to the vague topology in $\Conf$) probability measure $\mu$ on $\Conf$, we define the (pre-)Dirichlet form
\[
\cE(F,G)\define\int_{\Conf}\mathbb D[F,G](\theta)\,\mu(d\theta),
\qquad F,G\in\mathcal D^{\mathrm{loc}}_\infty,
\]
and also,
\[
\mathcal D_\infty\define\{F\in\mathcal D^{\mathrm{loc}}_\infty\cap L^2(\Conf,\mu):\ \cE(F,F)<\infty\}.
\]
When \((\cE,\mathcal{D}_\infty)\) is closable on \(L^2(\Conf,\mu)\), we denote
its closure by \((\cE,\cD)\).
 Moreover, if $(\cE,\cD)$ is quasi-regular, there exists an associated $\mu$-symmetric diffusion process $(\mathbf{X}_t)_{t\ge 0}$ on $\Conf$ under the probability measure~$\mathbf{P}_\gamma$  such that~$\mathbf{X}_0=\gamma$ $\mathbf{P}_\gamma$~almost surely.
 Roughly speaking, the quasi-regularity imposes that the domain~$\cD$ contains sufficiently many continuous functions, which is a necessary and sufficient condition for Dirichlet forms to be associated with Markov processes having a topological regularity of paths. We refer the readers to~Chen--Fukushima~\cite[Definition~1.3.8]{ChenFukushima2012} or Ma--Röckner~\cite[Chapter~IV, Section~3]{MaRoe90}. 
 A mild sufficient condition for the quasi-regularity of the closure~$(\cE, \cD)$ was given e.g.,~in~\cite[Theorem 1]{Osa96}.

 Under the quasi-regularity assumption, we write~$(\mathsf{T}_t)_{t\ge 0}$ for the transition semigroup of $(\mbfX_t)_{t \ge 0}$.  The relation between~the Dirichlet form~\((\cE,\cD)\) and the diffusion $(\mathbf{X}_t)_{t\ge 0}$ becomes more explicit when using the infinitesimal generator~$(\mssL, \mathcal D(\mssL))$: by the standard Riesz representation theorem, there exists a self-adjoint operator~$(\mssL, \mathcal D(\mssL))$ on $L^2(\Conf, \mu)$ such that
 $$\cE(u, v)= -(\mssL u, v) \qquad u \in \mathcal D(\mssL), \quad v \in \cD.$$
Then the transition semigroup satisfies $\mssT_t=e^{\mssL t}$,
namely, $\mssL$ is the infinitesimal generator of $(\mssT_t)_{t \ge 0}$ when $(\mssT_t)_{t \ge 0}$ is regarded as acting on~$L^2(\Conf, \mu)$. 

One notable case constructed via Dirichlet form is the infinite-dimensional Dyson Brownian motion (\cite{Osa12, Tsa16, OsaTan20}):
\begin{equation}\label{BulkDyson}
\mathrm{d}\mathsf{x}_i(t)=\mathrm{d}\mathsf{w}_i(t)+\frac{\beta}{2} \lim_{k \to \infty}\sum_{j:j \neq i, |j-i| \le k} \frac{1}{\mathsf{x}_i(t)-\mathsf{x}_j(t)}\mathrm{d}t
, \ \ i \in \mathbb{Z}.
\end{equation}
The solution is associated with the diffusion corresponding to $(\cE,\cD)$ with reversible measure~$\mu=\Sine_\beta$ for $\beta=1,2,4$. For general $\beta \ge 1$, the strong solutions to~\eqref{BulkDyson} has been constructed in \cite{Tsa16} based on the SDE approach. 
For work on equations for other long-range interaction models see \cite{Osa12, HonOsa15, OsaTan20, OsaTan14,AM24,AM26}.

\paragraph{Capacity} Assume that 
\((\cE,\mathcal{D}_\infty)\) is closable and take the closure~$(\cE, \cD)$. 
For an open set $O\subset\Conf$ define
\[
\mathcal{L}_O\define\{F\in\mathcal{D}:\ F\ge 1\ \text{$\mu$-a.e.~on }O\},
\qquad
\Cap(O)\define
\begin{cases}
\inf_{F\in\mathcal{L}_O}\bigl(\cE(F,F)+\|F\|_{L^2(\mu)}^2\bigr), & \mathcal{L}_O\neq \emptyset,\\
\infty, & \mathcal{L}_O=\emptyset.
\end{cases}
\]
For arbitrary $A\subset\Conf$, set $\Cap(A)\define\inf_{A\subset O\text{ open}}\Cap(O)$.
This is the outer capacity generated by the~$\mathcal D$; in particular it is monotone and countably subadditive.
We say a statement $\mathcal{A}(\gamma)$ holds for quasi-every $\gamma\in \Conf$, if it holds for all $\gamma \in \Conf$ except for a set $\mathfrak{A}$ with $\Cap(\mathfrak{A})=0$.

\subsection{Main results}
\paragraph{Non-collision criterion} We first give our non-collision criterion. For \(n\ge 1\) and \(\theta=\sum_i \delta_{x_i}\in\Conf\), define
\[
\theta^{[n]} \define\sum_{i_1,\dots,i_n\ \text{all distinct}} \delta_{(x_{i_1},\dots,x_{i_n})}
\quad\text{on } \mathbb{R}^n.
\]
Equivalently, for every nonnegative Borel function \(f\) on \(\mathbb{R}^n\),
\[
\theta \mapsto \theta^{[n]}(f)\define\int_{\mathbb{R}^n} f\, d\theta^{[n]}
=
\sum_{i_1,\dots,i_n\ \text{all distinct}} f(x_{i_1},\dots,x_{i_n}).
\]

The following is our criterion for non-collision. 

\begin{assumption}\label{ass:diag_repulsion} Suppose that $(\cE, \cD_\infty)$ is closable with the closure~$(\cE, \cD)$. Furthermore, assume that  
there exist locally integrable symmetric functions $\rho_2:\R^2\to[0,\infty)$ and $\rho_3:\R^3\to[0,\infty)$ such that, for every nonnegative Borel function $f_2$ on $\R^2$ and $f_3$ on $\R^3$,
\begin{align}
\int_{\Conf} \theta^{[2]}(f_2)\,\mu(d\theta)
&=\int_{\R^2} f_2(x,y)\rho_2(x,y)\,dx\,dy,
\label{eq:rho2-campbell}
\\
\int_{\Conf} \theta^{[3]}(f_3)\,\mu(d\theta)
&=\int_{\R^3} f_3(x,y,z)\rho_3(x,y,z)\,dx\,dy\,dz.
\label{eq:rho3-campbell}
\end{align}
Moreover, for every bounded open interval $J\subset\R$ there exists $C_J<\infty$ and $\alpha_J>0$ such that
\begin{equation}\label{eq:rho2-repulsion}
\rho_2(x,y)\le C_J |x-y|,
\end{equation}
for all $x,y\in J$ with $|x-y|\le 1$, and
\begin{equation}\label{eq:rho3-repulsion}
\rho_3(x,y,z)\le C_J (|x-y| \vee |x-z| \vee |y-z|)^{\alpha_J},
\end{equation}
for all $x,y,z\in J$ such that $|x-y| \vee |x-z| \vee |y-z|\le 1$. 
\end{assumption}


\begin{thm}[Non-collision]\label{thm:noncollision}
Under Assumption~\ref{ass:diag_repulsion}, 
 $$\Cap(\mathscr{C})=0.$$ 
If $(\cE, \cD)$ is quasi-regular, then 
$$\mathbf P_\gamma[\tau_{\mathscr{C}}<\infty]=0, \quad \text{quasi-every~$\gamma$},$$
where $\tau_{\mathscr{C}}$ is the first hitting time of the collision set~$\mathscr{C}$ for the associated diffusion process $(\mathbf{X}_t)_{t\ge 0}$ starting at the initial deterministic configuration~$\gamma$.
\end{thm}

\paragraph{Collision criterion} We now turn to our collision criterion. We write $\mu_{I_r^c}\define (\pi_{I_r^c})_\#\mu(\cdot)=\mu(\pi_{I_r^c}^{-1}(\cdot))$ for the law of the exterior configuration (i.e., the push-forward of $\mu$ by~$\pi_{I_r^c}$) and $\mu_{I_r}^{\eta}$ for a regular conditional law of $\theta|_{I_r}$ given the exterior configuration~$\theta|_{I_r^c}=\eta$.  Define $\Conf^k(I_r)=\{\theta \in \Conf(I_r): \theta(I_r)=k\}$. We write~$\mu_{I_r, k}^{\eta}= \mu_{I_r}^{\eta}|_{\Conf^k(I_r)}$.
For \(k\ge1\), we say that the \(k\)-particle sector of
\(\mu_{I_r}^{\eta}\) admits the labelled symmetric density
\(p_{r,k}^{\eta}(x)\) if \((x, \eta) \mapsto p_{r,k}^{\eta}(x)\) with $x=(x_1, \ldots, x_k)$ is jointly measurable for every $r>0$ and $k \in \N_0$, and for every non-negative Borel function
\(F\) on \(\Theta^k(I_r)\),
\begin{align} \label{d:FD}
    \int_{\Theta^k(I_r)}
    F(\xi)\,\mu_{I_r}^{\eta}(d\xi)
    =
    \frac1{k!}
    \int_{I_r^k}
    F\!\left(\sum_{i=1}^k\delta_{x_i}\right)
    p_{r,k}^{\eta}(x_1,\ldots,x_k)\,dx_1\cdots dx_k .
\end{align}
For \(k=0\), the sector consists of the empty configuration.
\begin{assumption}[Local collision lower bound]\label{ass:lqg_collision}
We suppose that
\begin{enumerate}[(1)]
\item \label{e:EXD}
for every $r>0$ and $k \in \N_0$ the jointly measurable density~$(x, \eta) \mapsto p_{r, k}^\eta(x)$ defined~in~\eqref{d:FD} exists, and the map $x  \mapsto p_{r, k}^\eta(x)$ is  bounded and  lower semi-continuous for $\mu_{I_r^c}$-a.e.~$\eta$.
\item  \label{e:EXD2}  there exist the following data:
\[
r>0,\qquad
J\subset I_r\define (-r,r)\ \text{open interval}
\]
\[
k\in \mathbb N\cup\{0\},\qquad
G\subset \Conf(I_r^c)\ \text{measurable},\qquad
K\subset I_r^k\ \text{compact},
\]
and constants
\[
0\le \beta<1,\qquad c_*>0,
\]
such that the following conditions hold: (when \(k=0\), we use the convention that \(K=\{\ast\}\), that \(K\) has measure one, and that
the variables \(z_1,\ldots,z_k\) are absent.)

\begin{enumerate}
\item \(
\mu_{I_r^c}(G)>0.
\)

\item The set \(K\) has positive \(k\)-dimensional Lebesgue measure if \(k\ge 1\).

\item For $\mu_{I_r^c}$-a.e.~\(\eta\in G\),  the density satisfies the lower bound
\[
p_{r,k+2}^{\eta}(x,y,z_1,\ldots,z_k)
\ge
c_* |x-y|^\beta \qquad \text{for a.e.
\(
(x,y,z_1,\ldots,z_k)\in J^2\times K.
\)}
\]
When \(k=0\), this means simply
\[
p_{r,2}^{\eta}(x,y)\ge c_*|x-y|^\beta
\qquad\text{for a.e. }(x,y)\in J^2.
\]
\end{enumerate}
\end{enumerate}
\end{assumption}

Under Assumption~\ref{ass:lqg_collision}, $(\cE,\mathcal{D}_\infty)$ is closable on $L^2(\Conf,\mu)$ (we recall it in Lemma~\ref{l:CLE}).
We say that $(\cE, \cD)$ is {\it irreducible} if $\cE(u,u)=0$ implies that $u$ is constant $\mu$-a.e.. Sufficient conditions for the irreducibility have been given in terms of number-rigidity and tail-triviality in~\cite{Suz23b}.

Our second main result is regarding collision.
\begin{thm}[Collision]\label{thm:lqg_collision}
Under Assumption~\ref{ass:lqg_collision}, 
\[
\Cap(\mathscr{C})>0.
\]
If the closure~$(\cE, \cD)$ is quasi-regular, then
$$\mathbf P_\mu[\tau_{\mathscr{C}}<\infty]>0,$$
where $\mathbf{P}_\mu$ denotes the law of $\mathbf{X}$ with a random initial distribution~$\mu$. If, furthermore, the closure~$(\cE, \cD)$ is irreducible, then the statement is strengthened as 
$$\mathbf P_\gamma[\tau_{\mathscr{C}}<\infty]=1 \qquad \text{q.e.~$\gamma$}.$$

\end{thm}

\subsection{Examples}

Our main example and initial motivation for this paper is to prove a complete characterisation for collision and non-collision for the diffusion associated with $\mu=\Sine_\beta$, which is known as corresponding to the infinite-dimensional Dyson Brownian motion \eqref{BulkDyson} for $\beta=1,2,4$ (\cite{Osa12, Tsa16, OsaTan20}).
We believe that this correspondence holds for every~$\beta>0$, which however remains open. 
The form~$(\cE,\mathcal{D}_\infty)$ is  closable and the closure is quasi-regular for every $\beta>0$, which will be seen in the proof of Theorem~\ref{DysonThm} below. We write $(\mathsf{X}^{(\beta)}_t)_{t\ge 0}$  for the associated diffusion and \(\Cap_\beta\) for the associated capacity. 
\begin{thm}[Dichotomy for $\Sine_\beta$]\label{DysonThm}
The form~$(\cE,\mathcal{D}_\infty)$ is  closable and the closure is quasi-regular for every $\beta>0$. Furthermore, 
\[\Cap_\beta(\mathscr C)=0 \qquad \text{if and only if} \qquad  \beta \ge 1.\]
 In particular, $\mathbf P_\gamma[\tau_{\mathscr{C}}<\infty]=0$ q.e.~$\gamma$ if and only if $\beta \ge 1$. In other words, the diffusion $(\mathsf{X}^{(\beta)}_t)_{t\ge 0}$ has no collisions for q.e.~initial configuration~$\gamma$ if and only if $\beta \ge 1$. 
 For $0<\beta<1$, therefore, we have $\Cap_\beta(\mathscr C)>0$ and that the collision happens with a positive probability:
 $$\mathbf P_\mu[\tau_{\mathscr{C}}<\infty]>0.$$
\end{thm}

\begin{remark} \normalfont \
\begin{itemize}
\item (Tail-triviality) If $\mu=\Sine_\beta$ is tail-trivial for $\beta <1$, i.e., $\mu(A) \in \{0,1\}$ for every tail event $A \in \cap_{r>0} \sigma[\pi_{I_r^c}]$, then $(\cE, \cD)$ is irreducible by \cite[Corollary I]{Suz23b} combined with the number-rigidity~\cite{DLR}. Thus, the collision indeed happens almost surely for $\beta<1$:
$$\mathbf P_\gamma[\tau_{\mathscr{C}}<\infty]=1 \qquad \text{q.e.~$\gamma$}.$$
The tail-triviality for $\Sine_\beta$ however remains open except for the $\beta=2$ case while we  believe the tail-triviality for every $\beta>0$.
\item (SDE approach) In \cite{Tsa16}, Tsai proved that the solutions to the infinite-dimensional Dyson Brownian motion \eqref{BulkDyson} do not have a collision for $\beta \ge 1$ based on the SDE approach. He also predicted in \cite[Remark 1.1]{Tsa16} that the collision is expected to occur for $0<\beta<1$. Although the rigorous construction of the strong solutions to  SDE~\eqref{BulkDyson} remains open in the regime~$0<\beta<1$, our result  supports his prediction from the Dirichlet-form perspective.   
\end{itemize}
\end{remark}
We expect that  our criteria should apply to a general class of limiting $\beta$-ensembles such as the $\beta$-Airy \cite{RRV} and $\beta$-Bessel \cite{RRbessel} point processes and give an analogous collision/non-collision dichotomy. 
We note that our non-collision criterion recovers the fundamental fact proven by Osada \cite{Osa04}: $\Cap(\mathscr C)=0$ if $\mu$ is determinantal associated with  locally Lipschitz correlation kernel~$\mathsf{K}$, see Proposition \ref{prop:DetermintalCap}.

\section{Proof of Theorem~\ref{thm:noncollision}}
\subsection{Auxiliary lemmas}

\begin{lemma}\label{l:h}
There exists a family  of even functions $h_\eps\in C_c^\infty(\R)$, $0<\eps<1/4$, such that
\begin{equation}\label{eq:heps-properties-intrinsic}
0\le h_\eps\le 1,
\qquad
h_\eps(t)=1\ \text{for }|t|\le \eps/2,
\qquad
\supp h_\eps\subset \{|t|<\sqrt\eps\},
\end{equation}
and
\begin{equation}\label{eq:heps-derivative-intrinsic}
|h_\eps'(t)|\le \frac{c_h}{|\log\eps|\,(|t|\vee \eps)}
\qquad (t\in\R),
\end{equation}
with $c_h>0$ independent of $\eps$.
\end{lemma}
\begin{proof}
Let $\chi\in C^\infty(\mathbb R)$ satisfy
\[
0\le \chi\le 1,\qquad
\chi(u)=1\ \text{for }u\le 0,\qquad
\chi(u)=0\ \text{for }u\ge 1.
\]
For $0<\varepsilon<1/4$, set
\[
L_\varepsilon\define \frac12|\log\varepsilon|,
\qquad
h_\varepsilon(t)\define
\begin{cases}
\chi\!\left(\dfrac{\log(2|t|/\varepsilon)}{L_\varepsilon}\right), & t\neq 0,\\[1.2ex]
1,& t=0.
\end{cases}
\]
Then $h_\varepsilon\in C_c^\infty(\mathbb R)$ is even and satisfies
\[
0\le h_\varepsilon\le 1,\qquad
h_\varepsilon(t)=1\ \text{for }|t|\le \varepsilon/2,\qquad
\operatorname{supp}h_\varepsilon\subset\{|t|<\sqrt\varepsilon\},
\]
and
\[
|h'_\varepsilon(t)|\le \frac{c_h}{|\log\varepsilon|(|t|\vee \varepsilon)}
\qquad (t\in\mathbb R)
\]
with $c_h=4\|\chi'\|_\infty$.
\end{proof}

Now, fix a bounded open interval $J_0\subset\R$ and let
\[
\mathscr{C}(J_0)\define\{\theta\in\Conf:\ \theta(\{x\})\ge 2\ \text{for some }x\in J_0\}.
\]
Choose a bounded open interval $J_1\subset\R$ such that $\overline{J_0} \subset J_1$. Let $\lambda\in C_c^\infty(\R)$ satisfy
\[
0\le \lambda\le 1,
\qquad
\lambda\equiv 1\ \text{on }J_0,
\qquad
\supp\lambda\subset J_1.
\]
 Let $\Phi\in C_b^\infty(\R)$ satisfy
\begin{equation}\label{eq:Phi-cutoff-intrinsic}
0\le \Phi\le 1,
\qquad
\Phi=0\ \text{on }[0,1/2],
\qquad
\Phi=1\ \text{on }[1,\infty).
\end{equation}
Define the local pair statistic
\begin{equation}\label{eq:S-eps-def-intrinsic}
S_\eps(\theta)\define\int_{\R^2}\lambda(x)\lambda(y)h_\eps(x-y)\,\theta^{[2]}(dx\,dy)
\end{equation}
and the bounded local smooth cutoff
\begin{equation}\label{eq:F-eps-def-intrinsic}
F_\eps(\theta)\define\Phi\bigl(S_\eps(\theta)\bigr).
\end{equation}

\begin{lemma}\label{lem:F-eps-admissible}
For every $\eps\in(0,1/4)$, the function $F_\eps$ belongs to $\mathcal D_\infty^{\mathrm{loc}}$.
Moreover $F_\eps(\theta)=1$ for every $\theta\in \mathscr{C}(J_0)$.
\end{lemma}

\begin{proof}
Fix arbitrary $u>0$, $n\ge 0$, and an exterior configuration $\eta\in\Conf(I_u^c)$. Since \(\lambda\) has compact support,
\[
    \eta|_{\operatorname{supp}\lambda}
    =
    \sum_{\ell=1}^{N}\delta_{y_\ell}
\]
for suitable \(y_1,\ldots,y_N\), where repetitions are allowed.
Thus, using the evenness of $h_\varepsilon$ by~Lemma~\ref{l:h},  the $(u,\eta)$-representative of $S_\eps$ on $I_u^n$ is
\begin{align*}
S_{\eps;u,\eta}^n(x)
&=
\sum_{\substack{1\le p,q\le n\\ p\neq q}} \lambda(x_p)\lambda(x_q)h_\eps(x_p-x_q)
+2\sum_{p=1}^n\sum_{\ell=1}^N \lambda(x_p)\lambda(y_\ell)h_\eps(x_p-y_\ell)
\\
&+\sum_{\substack{1\le \ell,m\le N\\ \ell\neq m}} \lambda(y_\ell)\lambda(y_m)h_\eps(y_\ell-y_m).
\end{align*}
This is a smooth symmetric function of $x\in I_u^n$, since $\lambda$ and $h_\eps$ are smooth. Therefore every finite-volume representative of $F_\eps=\Phi\circ S_\eps$ is smooth, and because $0\le F_\eps\le 1$, we obtain $F_\eps\in\mathcal D_\infty^{\mathrm{loc}}$.

Finally, if $\theta\in \mathscr{C}(J_0)$, then there exists $x\in J_0$ with $\theta(\{x\})\ge 2$. Hence
\[
\theta^{[2]}(\{(x,x)\})=\theta(\{x\})(\theta(\{x\})-1)\ge 2.
\]
Since $\lambda(x)=1$ and $h_\eps(0)=1$, the definition \eqref{eq:S-eps-def-intrinsic} gives $S_\eps(\theta)\ge 2$. By \eqref{eq:Phi-cutoff-intrinsic}, $F_\eps(\theta)=1$.
\end{proof}

\begin{prop}\label{lem:local-collision-cap-zero}
Under Assumption~\ref{ass:diag_repulsion}, 
\[
\Cap\bigl(\mathscr{C}(J_0)\bigr)=0.
\]
\end{prop}

\begin{proof}
First, let
\[
U_\eps\define\{\theta\in\Conf:\ F_\eps(\theta)>1/2\}.
\]
By Lemma~\ref{lem:F-eps-admissible}, $U_\eps$ is open and $\mathscr{C}(J_0)\subset U_\eps$. 
Moreover $2F_\eps\ge 1$ on $U_\eps$.
To estimate the capacity, we first give an estimate for the $L^2$-term of $F_\eps$, then for the energy part of $F_\eps$.

\medskip
\noindent\textbf{1.~$L^2$-term.}
Since $0\le F_\eps\le 1$ and $\Phi(t)=0$ for $t\le 1/2$,
\[
F_\eps(\theta)^2\le \mathbf 1_{\{S_\eps(\theta)>1/2\}}\le 2S_\eps(\theta).
\]
Therefore, by \eqref{eq:rho2-campbell},
\begin{align}
\|F_\eps\|_{L^2(\mu)}^2
&\le 2\int_{\Conf} S_\eps(\theta)\,\mu(d\theta)
\notag
\\
&=2\iint_{\R^2} \lambda(x)\lambda(y)h_\eps(x-y)\rho_2(x,y)\,dx\,dy.
\label{eq:Feps-L2-rho2}
\end{align}
Since $\supp\lambda\subset J_1$ and $\supp h_\eps\subset \{|t|<\sqrt\eps\}\subset\{|t|<1\}$, \eqref{eq:rho2-repulsion} gives
\[
\|F_\eps\|_{L^2(\mu)}^2
\le C_1(J_1)\iint_{\substack{x,y\in J_1\\ |x-y|<\sqrt\eps}} |x-y|\,dx\,dy
\le C_2(J_1)\int_0^{\sqrt\eps} t\,dt.
\]
Hence, we have
\begin{equation}\label{eq:Feps-L2-to-0-rho}
\|F_\eps\|_{L^2(\mu)}\longrightarrow 0
\qquad (\eps\downarrow 0).
\end{equation}

\medskip
\noindent\textbf{2.~Energy term.}
We set
\[
A_\eps(x,y)\define\mathbf 1_{\{x,y\in J_1,\ |x-y|<\sqrt\eps\}}
\left(1+\frac{1}{|\log\eps|\,(|x-y|\vee\eps)}\right).
\]
Choose \(R>0\) such that
\[
    \overline{J_1}\subset I_R.
\]
Since \(S_\varepsilon\) is \(\sigma[\pi_R]\)-measurable, on the
\(n\)-particle sector of \(I_R\), with coordinate
\(x=(x_1,\ldots,x_n)\), its representative is
\[
    S_{\varepsilon,R}^{n}(x)
    =
    \sum_{\substack{1\le p,q\le n\\p\ne q}}
    \lambda(x_p)\lambda(x_q)h_\varepsilon(x_p-x_q).
\]
Using the evenness of $h_\varepsilon$ (hence the derivative $h_\varepsilon'$ is odd), a direct differentiation gives
\[
\partial_{x_i}S_{\eps,R}^n(x)
=2\sum_{j\neq i}
\Bigl(\lambda'(x_i)\lambda(x_j)h_\eps(x_i-x_j)
+\lambda(x_i)\lambda(x_j)h_\eps'(x_i-x_j)\Bigr).
\]
Since $\supp\lambda\cup\supp\lambda'\subset J_1$, using the boundedness of $\lambda$ and $\lambda'$ together with \eqref{eq:heps-properties-intrinsic}--\eqref{eq:heps-derivative-intrinsic}, we obtain
\begin{equation}\label{eq:dSi-bound}
|\partial_{x_i}S_{\eps,R}^n(x)|\le C_3\sum_{j\neq i} A_\eps(x_i,x_j).
\end{equation}
Hence, for a constant $C_4$ independent of $n$ and $\eps$,
\begin{align}
\sum_{i=1}^n |\partial_{x_i}S_{\eps,R}^n(x)|^2
&\le C_4\sum_{i=1}^n\left(\sum_{j\neq i} A_\eps(x_i,x_j)\right)^2
\notag
\\
&\le C_4\sum_{i=1}^n\sum_{j\neq i} A_\eps(x_i,x_j)^2
+C_4\sum_{i=1}^n\sum_{\substack{j,k\neq i\\ j\neq k}} A_\eps(x_i,x_j)A_\eps(x_i,x_k).
\label{eq:dSi-square-split}
\end{align}
Since $F_\eps=\Phi(S_\eps)$ and $\Phi'$ is bounded, for $\theta\in\Conf^n_R$ with sector coordinate $x=(x_1,\dots,x_n)$ we have
\begin{equation}\label{eq:DF-via-sectorwise-S}
\mathbb D[F_\eps](\theta)
=
\frac12\sum_{i=1}^n |\partial_{x_i}F_{\eps,R}^n(x)|^2
\le
\frac{C_5}{2}\sum_{i=1}^n |\partial_{x_i}S_{\eps,R}^n(x)|^2.
\end{equation}
Combining \eqref{eq:DF-via-sectorwise-S} with \eqref{eq:dSi-square-split} and integrating over $\Conf$ yields
\begin{equation}\label{eq:E-by-factorial-parts}
\cE(F_\eps,F_\eps)
\le C_5'
\left[
\int_{\Conf}\theta^{[2]}(A_\eps^2)\,\mu(d\theta)
+
\int_{\Conf}\theta^{[3]}\bigl(A_\eps(x,y)A_\eps(x,z)\bigr)\,\mu(d\theta)
\right]
\end{equation}
for a constant $C_5'>0$ independent of $\eps$. We estimate the two contributions separately.

\smallskip
\noindent\textbf{2-(a).~Two-point part.}
By \eqref{eq:rho2-campbell}, we have
\begin{align}
\int_{\Conf}\theta^{[2]}(A_\eps^2)\,\mu(d\theta)
&=\iint_{\R^2} A_\eps(x,y)^2\rho_2(x,y)\,dx\,dy
\notag
\\
&\le C_6(J_1)
\left[
\int_0^{\sqrt\eps} t\,dt
+\frac1{|\log\eps|^2}
\left(
\int_0^{\eps}\frac{t}{\eps^2}\,dt
+\int_{\eps}^{\sqrt\eps} t^{-1}\,dt
\right)
\right].
\label{eq:two-point-energy-bound}
\end{align}
Observe that the right-hand side tends to $0$ as $\eps\downarrow 0$. 

\smallskip
\noindent\textbf{2-(b).~Three-point part.}
If \(A_\varepsilon(x,y)A_\varepsilon(x,z)\ne0\), then
\[
    x,y,z\in J_1,\qquad |y-x|<\sqrt\varepsilon,\qquad |z-x|<\sqrt\varepsilon .
\]
Hence
\[
    |y-z|\le |y-x|+|z-x|<2\sqrt\varepsilon.
\]
Taking \(0<\varepsilon<1/4\), all pairwise distances are less than \(1\). Therefore
Assumption~\ref{ass:diag_repulsion} applied on \(J_1\) gives
\[
    \rho_3(x,y,z)
    \le
    C_{J_1}
    \left(
        |x-y|\vee |x-z|\vee |y-z|
    \right)^{\alpha_{J_1}}.
\]
Writing \(u=y-x\) and \(v=z-x\), we have
\[
    |y-z|=|u-v|\le2(|u|\vee|v|),
\]
and hence, after increasing the constant,
\[
    \rho_3(x,x+u,x+v)
    \le
    C_7(J_1)(|u|\vee|v|)^{\alpha_{J_1}}.
\]
Thus, we obtain
\begin{align}
&\int_{\Conf}
\sum_{\substack{x,y,z\in\theta\\ \mathrm{pairwise\ distinct}}}
A_\eps(x,y)A_\eps(x,z)\,\mu(d\theta)
\notag\\
&\qquad\le
C_7(J_1)
\int_{|u|<\sqrt\eps}
\int_{|v|<\sqrt\eps}
\left(
1+\frac{1}{|\log\eps|(|u|\vee\eps)}
\right)
\left(
1+\frac{1}{|\log\eps|(|v|\vee\eps)}
\right)
(|u|\vee |v|)^{\alpha_{J_1}}\,du\,dv .
\label{eq:three-point-energy-bound}
\end{align}
The right-hand side tends to \(0\) as \(\eps\downarrow0\). Indeed, after
splitting the product, the only potentially singular term is bounded by
\[
\frac{C}{|\log\eps|^2}
\int_0^{\sqrt\eps}\int_0^{\sqrt\eps}
\frac{(u\vee v)^{\alpha_{J_1}}}{(u\vee\eps)(v\vee\eps)}\,du\,dv .
\]
Using symmetry and integrating over \(0<v<u<\sqrt\eps\), this is at most
\[
\frac{C}{|\log\eps|^2}
\int_0^{\sqrt\eps}
\frac{u^{\alpha_{J_1}}}{u\vee\eps}
\left(\int_0^u\frac{dv}{v\vee\eps}\right)du
\le
\frac{C}{|\log\eps|^2}
\left[
\eps^{\alpha_{J_1}}+
\int_\eps^{\sqrt\eps}
u^{\alpha_{J_1}-1}\bigl(1+\log(u/\eps)\bigr)\,du
\right],
\]
which tends to \(0\) for every \(\alpha_{J_1}>0\). The remaining terms are easier
and vanish as well.

\paragraph{Conclusion} Now, by the definition of capacity,
\begin{equation}\label{eq:cap-local-by-Feps}
\Cap\bigl(\mathscr{C}(J_0)\bigr)
\le \Cap(U_\eps)
\le 4\bigl(\cE(F_\eps,F_\eps)+\|F_\eps\|_{L^2(\mu)}^2\bigr).
\end{equation}
Combining \eqref{eq:E-by-factorial-parts}, \eqref{eq:two-point-energy-bound}, and \eqref{eq:three-point-energy-bound}, we conclude that
\begin{equation}\label{eq:Feps-E-to-0-rho}
\cE(F_\eps,F_\eps)\longrightarrow 0
\qquad (\eps\downarrow 0).
\end{equation}
Finally, \eqref{eq:cap-local-by-Feps}, \eqref{eq:Feps-L2-to-0-rho}, and \eqref{eq:Feps-E-to-0-rho} imply
\[
\Cap\bigl(\mathscr{C}(J_0)\bigr)=0,
\]
which completes the proof.
\end{proof}

\begin{proof}[Proof of Theorem~\ref{thm:noncollision}]
For each $m\in\N$, let
\[
\mathscr{C}_m\define \mathscr{C}((-m,m)).
\]
so that,
\[
\mathscr{C}=\bigcup_{m=1}^{\infty} \mathscr{C}_m.
\]
By Proposition~\ref{lem:local-collision-cap-zero}, each $\mathscr{C}_m$ has capacity $0$ and thus by countable subadditivity we obtain $\Cap(\mathscr{C})=0$. 
Now suppose that $(\cE, \cD)$ is quasi-regular. Since~$\Cap(\mathscr{C})=0$,  
the set~$\mathscr{C}$ is polar. Therefore by e.g. \cite[Theorem~3.1.3]{ChenFukushima2012}, 
\[
\mathbf P_\gamma[\tau_\mathscr{C}<\infty]=0,
\qquad\text{for quasi-every }\gamma,
\]
as desired.
\end{proof}

\section{Proof of Theorem~\ref{thm:lqg_collision}}\label{sec:colision}
Recall that for a function \(f\) on \(\Conf(I_r)\), we write~\(f_r^\ell\)  for a symmetric
\(\ell\)-particle sector representative of $f$ on \(I_r^\ell\) (recall~\eqref{e:SR}).
We consider the form~
\begin{align} \label{e:TD}
\cE_{I_r}^{\eta, \ell}(f,f) \define \frac{1}{2\ell!}
\int_{I_r^\ell}
\sum_{i=1}^{\ell} |\partial_{x_i}f_r^\ell(x)|^2\,p_{r,\ell}^{\eta}(x)\,dx \qquad f_r^\ell \in C^\infty(\overline{I_r^\ell}), 
\end{align}
on $L^2(I_r^\ell,\frac1{\ell!}p_{r,\ell}^{\eta}(x)\,dx)$. We set~$\cE_{I_r}^{\eta,0} \define 0.$

Since  $p_{r,\ell}^{\eta}$ is lower semicontinuous by \ref{e:EXD}~of~Assumption~\ref{ass:lqg_collision}, this form is closable; see
\cite[Lemma~3.2]{Osa96}.
We denote by $\mathcal D_r^{\eta, \ell}$ the closure. 
We define
\[
\cE_{I_r}^{\eta}(f,f)
\define
\sum_{\ell=0}^{\infty} \cE_{I_r}^{\eta, \ell}(f,f).
\]
Since the form $\cE^\eta_{I_r}$ is the sum of $\cE^{\eta, \ell}_{I_r}$ along $\ell$, it is closable on 
$$\mathcal D_{r,\infty}^{\eta}
\define
\left\{
    f:
    f_\ell\in C^\infty(\overline{I_r^\ell}) \ \text{for every }\ell \in \N_0,\
    \sum_{\ell=0}^{\infty}
    \frac1{\ell!}
    \int_{I_r^\ell}|f_\ell|^2p_{r,\ell}^{\eta}\,dx<\infty,\
    \sum_{\ell=0}^{\infty}
    \cE_{I_r}^{\eta,\ell}(f_\ell,f_\ell)<\infty
\right\}$$
see, e.g., \cite[Proposition 3.7, Chapter I]{MaRoe90}. We denote its closure by $\mathcal D_r^\eta$.
We also write \(\cE_r\) by
\[
\cE_r(F,F)
\define
\int_{\Conf(I_r^c)}
\cE_{I_r}^{\eta}(F_r^\eta,F_r^\eta)\,
\mu_{I_r^c}(d\eta),
\qquad
F_r^\eta(\xi)\define F(\eta+\xi),\quad \xi\in\Conf(I_r).
\]
As the form $\cE_r$ is what is called {\it superposition} of $\cE_{I_r}^\eta$ in terms of $\eta$, it is closable on $\mathcal D_\infty$, see \cite[Chapter V, Proposition 3.1.1]{BouleauHirsch1991}.
We denote its closure by \((\mathcal E_r,\mathcal D_r)\).
\begin{lemma}\label{l:CLE}
Under \ref{e:EXD}~of~Assumption~\ref{ass:lqg_collision}, 
the form~ $\cE$ is closable in $\mathcal D_\infty$. Furthermore, the closures satisfy $\cE_r \le \cE$  for every $r>0$.
\end{lemma}
\begin{proof}

Note that the family~$r \mapsto \cE_r$ is monotone  increasing on $\mathcal D_\infty$ by the monotonicity of the square field~\eqref{lem:square-field-compat} (or see~e.g.,~\cite[Lemma~2.2]{Osa96}), and also note that  the form~$(\cE, \mathcal D_\infty)$  is the monotone limit form of~the sequence~$m \mapsto (\cE_{m}, \mathcal D_\infty)$ for $m \in \N$. Thus, the form~$(\cE, \mathcal D_\infty)$ is closable, by \cite[Proposition 3.7, Chapter I]{MaRoe90} and the closures satisfy $\cE_r \le \cE$ for every $r>0$.
\end{proof}

Take $0 \le \beta<1$ and a bounded open interval~$J$. Set
\(
    w(x,y)\define |x-y|^\beta 
\)
and define
\[
    H^1_w(J^2)
    \define
    \overline{C^\infty(\bar J^2)}^{\|\cdot\|_{H^1_w(J^2)}},
\]
where
\[
    \|v\|_{H^1_w(J^2)}^2
    \define
    \int_{J^2}
    \left(
        |v|^2+|\partial_x v|^2+|\partial_y v|^2
    \right)|x-y|^\beta\,dx\,dy .
\]
\begin{lem}\label{lem:weighted_two_particle}
Let \(0\le \beta<1\), and let \(L\Subset J\) be a compact interval of positive
Lebesgue measure contained in the bounded open interval \(J\). There exists a constant \(c_{L,J,\beta}>0\) such that the following holds: if
\(v \in H^1_w(J^2)\) satisfies  that there exists an open neighborhood \(W\subset J^2\) of
\[
    \Delta_L\define\{(x,x):x\in L\}
\]
such that
\[
    v\ge 1
    \qquad
    |x-y|^\beta dxdy\text{-a.e. on }W,
\]
then
\[
    \int_{J^2}
    \left(
        |v|^2+|\partial_x v|^2+|\partial_y v|^2
    \right)|x-y|^\beta\,dx\,dy
    \ge c_{L,J,\beta}.
\]
\end{lem}

\begin{proof}
Choose \(R>0\) such that
\[
    m+t\in J,\qquad m-t\in J
\]
for every \(m\in L\) and every \(|t|<R\). Since \(\Delta_L\) is compact and
contained in \(W\), after decreasing \(R\) if necessary, there exists
\(\varepsilon_0\in(0,R)\) such that
\[
    (m+t,m-t)\in W
    \qquad
    (m\in L,\ |t|<\varepsilon_0).
\]

We first prove a one-dimensional estimate. Let \(f\) satisfy
\[
    \int_{-R}^R
    \left(
        |f|^2+|f'|^2
    \right)|t|^\beta\,dt<\infty
\]
and
\[
    f\ge1
    \qquad
    |t|^\beta dt\text{-a.e. on }(-\varepsilon_0,\varepsilon_0).
\]
Then there exists \(c_{1d}=c_{1d}(R,\beta)>0\) such that
\begin{align} \label{e:1D}
    \int_{-R}^R
    \left(
        |f|^2+|f'|^2
    \right)|t|^\beta\,dt
    \ge c_{1d}.
\end{align}
Indeed, since \(\beta<1\),
\[
    \int_{-R}^R t^{-\beta}\,dt<\infty.
\]
Thus, by Cauchy--Schwarz, 
\[
\begin{aligned}
    \int_{-R}^{R}|f(t)|\,dt
    &\le
    \left(
        \int_{-R}^{R}|f(t)|^2|t|^\beta\,dt
    \right)^{1/2}
    \left(
        \int_{-R}^{R}|t|^{-\beta}\,dt
    \right)^{1/2}
    <\infty,\\
    \int_{-R}^{R}|f'(t)|\,dt
    &\le
    \left(
        \int_{-R}^{R}|f'(t)|^2|t|^\beta\,dt
    \right)^{1/2}
    \left(
        \int_{-R}^{R}|t|^{-\beta}\,dt
    \right)^{1/2}
    <\infty.
\end{aligned}
\]
Hence \(f\in W^{1,1}((-R,R))\), and \(f\) admits an absolutely
continuous representative on \([-R,R]\).
Since \(f\ge1\) a.e.~near \(0\), this
representative satisfies \(f(0)\ge1\).
Put
\[
    A\define\int_0^R |f'(t)|^2t^\beta\,dt,
    \qquad
    B\define\int_0^R |f(t)|^2t^\beta\,dt .
\]
Choose \(a\in[R/2,R]\) such that
\[
    |f(a)|^2
    \le
    \frac2R\int_{R/2}^R |f(t)|^2\,dt
    \le C(R,\beta)B .
\]
Then, noting that $f(a)-f(0)=\int_0^a f'(t)\,dt$ implies
$$|f(0)| \le |f(a)| + \int_0^a|f'(t)|\,dt \le  |f(a)| + \int_0^R|f'(t)|\,dt,$$
we have
\[
\begin{aligned}
    1
    \le |f(0)|
    &\le |f(a)|+\int_0^R |f'(t)|\,dt                                    \\
    &\le C(R,\beta)B^{1/2}
       +
       \left(
           \int_0^R |f'(t)|^2t^\beta\,dt
       \right)^{1/2}
       \left(
           \int_0^R t^{-\beta}\,dt
       \right)^{1/2}                                                       \\
    &\le C'(R,\beta)(A+B)^{1/2}.
\end{aligned}
\]
Therefore \(A+B\ge c_{1d}(R,\beta)>0\), proving the one-dimensional estimate~\eqref{e:1D}.

\smallskip
Now let \(v\in H^1_w(J^2)\). Use the change of variables
\[
    x=m+t,\qquad y=m-t .
\]
For a.e. \(m\in L\), the slice
\[
    f_m(t)\define v(m+t,m-t),\qquad |t|<R,
\]
belongs to the one-dimensional weighted Sobolev space associated with
\[
    \int_{-R}^R
    \left(
        |f|^2+|f'|^2
    \right)|t|^\beta\,dt ,
\]
and its weak derivative satisfies
\[
    f'_m(t)
    =
    \partial_xv(m+t,m-t)-\partial_yv(m+t,m-t)
\]
for a.e. \(t\in(-R,R)\). Moreover, by the assumption \(v\ge1\) on \(W\),
\[
    f_m(t)\ge1
    \qquad
    |t|^\beta dt\text{-a.e. on }(-\varepsilon_0,\varepsilon_0)
\]
for a.e.~\(m\in L\). Applying the one-dimensional estimate~\eqref{e:1D} and integrating over
\(m\in L\), we obtain
\[
    |L|c_{1d}(R,\beta)
    \le
    \int_L\int_{-R}^R
    \left(
        |f_m(t)|^2+|f'_m(t)|^2
    \right)|t|^\beta\,dt\,dm .
\]
Since
\[
    |f'_m(t)|^2
    \le
    2\left(
        |\partial_xv|^2+|\partial_yv|^2
    \right)(m+t,m-t),
\]
and since \(dxdy=2\,dmdt\), while
\[
    |x-y|^\beta=2^\beta |t|^\beta,
\]
the last display is bounded above by a constant multiple, depending only on
\(\beta\), of
\[
    \int_{J^2}
    \left(
        |v|^2+|\partial_xv|^2+|\partial_yv|^2
    \right)|x-y|^\beta\,dx\,dy .
\]
This proves the assertion.
\end{proof}

\begin{proof}[Proof of Theorem~\ref{thm:lqg_collision}]
Let the data in Assumption~\ref{ass:lqg_collision} be fixed. Thus we have
\[
    r>0,\qquad J\subset I_r,\qquad k\in\mathbb N\cup\{0\},
\]
a measurable set
\[
    G\subset\Conf(I_r^c),
    \qquad
    \mu_{I_r^c}(G)>0,
\]
a compact set
\[
    K\subset I_r^k,
\]
and constants
\[
    0\le\beta<1,\qquad c_\ast>0.
\]
When \(k=0\), we use the convention that \(K=\{\ast\}\), \(|K|=1\), and that
all integrations over \(K\) are omitted. Set
\[
    N\define k+2.
\]

After replacing \(G\) by a measurable subset of the same positive
\(\mu_{I_r^c}\)-measure if necessary, we may assume that the sector density and the
lower bound in Assumption~\ref{ass:lqg_collision} hold for every \(\eta\in G\).
Choose a compact interval
\[
    L\Subset J
\]
with positive Lebesgue measure. By inner regularity, choose a compact set
\(
    G_0\subset G
\)
such that
\[
    \mu_{I_r^c}(G_0)>0.
\]

For \((\eta,x,y,z)\in G_0\times J^2\times K\), where
\(z=(z_1,\ldots,z_k)\) when \(k\ge1\), define
\[
    \Phi(\eta,x,y,z)
    \define
    \eta+\delta_x+\delta_y+\sum_{j=1}^k\delta_{z_j},
\]
and
\[
    \Phi_0(\eta,x,z)
    \define
    \eta+2\delta_x+\sum_{j=1}^k\delta_{z_j}.
\]
When \(k=0\), the sums are absent.

Define
\[
    \mathscr C_0
    \define
    \Phi_0(G_0\times L\times K).
\]
The map \(\Phi_0\) is continuous for the vague topology, and
\(G_0\times L\times K\) is compact. Hence \(\mathscr C_0\) is compact. Moreover,
\(
    \mathscr C_0\subset \mathscr C.
\)
It is therefore enough to prove
\[
    \Cap(\mathscr C_0)>0.
\]

Let \(O\subset\Conf\) be open and suppose that
\(
    \mathscr C_0\subset O.
\)
Let \(u\in\mathcal D\) be admissible for \(O\), namely
\[
    u\ge1
    \qquad
    \mu\text{-a.e. on }O.
\]
Since
\(
    \Phi(G_0\times\Delta_L\times K)
    =
    \mathscr C_0
    \subset O,
\)
where
\(
    \Delta_L\define\{(x,x):x\in L\},
\)
we have
\[
    G_0\times\Delta_L\times K
    \subset
    \Phi^{-1}(O).
\]
The set \(G_0\times\Delta_L\times K\) is compact and \(\Phi^{-1}(O)\) is open.
Hence, by compactness, there exists an open neighborhood \(W\subset J^2\) of
\(\Delta_L\) such that
\begin{equation}\label{eq:lqg_tube}
    G_0\times W\times K
    \subset
    \Phi^{-1}(O).
\end{equation}

We next pass from \(u\in\mathcal D\) to the finite-volume sector representatives.
By~Lemma~\ref{l:CLE}, the form
\((\mathcal E_r,\mathcal D_r)\) satisfies
\begin{equation}\label{eq:lqg_local_domination}
    \mathcal D\subset\mathcal D_r,
    \qquad
    \mathcal E_r(u,u)\le \mathcal E(u,u)
    \qquad
    (u\in\mathcal D).
\end{equation}
For \(\mu_{I_r^c}\)-a.e.~\(\eta\), denote by
\(
    u_r^{\eta,N}:I_r^N\to\mathbb R
\)
the \(N\)-particle sector representative of \(u\) in \(I_r\). Thus
\[
    u\left(\eta+\sum_{i=1}^N\delta_{x_i}\right)
    =
    u_r^{\eta,N}(x_1,\ldots,x_N)
\]
for the corresponding conditional measure-a.e. \((\eta,x_1,\ldots,x_N)\). We prove the following auxiliary statement.

\medskip
\noindent\textbf{Claim.}
\(\mu_{I_r^c}(d\eta)\,dz\)-a.e.
\(
    (\eta,z)\in G_0\times K,
\)
the two-variable slice
\[
    v_{\eta,z}(x,y)
    \define
    u_r^{\eta,N}(x,y,z_1,\ldots,z_k),
    \qquad (x,y)\in J^2,
\]
belongs to \(H^1_w(J^2)\), where
\(
    w(x,y)\define|x-y|^\beta
\)
and \[
    v_{\eta,z}(x,y)\ge1
\]
for \(dxdy\)-a.e. \((x,y)\in W\).
\begin{proof}[Proof of Claim]
Fix such an exterior configuration \(\eta\). By the domain characterization for superposed Dirichlet forms
\cite[Proposition~3.1.1]{BouleauHirsch1991} (see also~\cite[Proposition 5.14]{LzDSSuz21}), for \(\mu_{I_r^c}\)-a.e. \(\eta\),
the finite-volume representative \(u_r^\eta\) belongs to the closed domain
\(\mathcal D_r^\eta\), and the \(N\)-particle sector representative
\(u_r^{\eta,N}\) belongs to \(\mathcal D_r^{\eta,N}\).
Thus, there exists
a sequence
\[
    f_m\in C^\infty(I_r^N)
\]
such that
\[
    f_m\to u_r^{\eta,N}
\]
in the \(N\)-sector graph norm. Hence
\begin{equation}\label{eq:lqg_sector_cauchy}
    \int_{I_r^N}
    \left(
        |f_m-f_n|^2
        +
        \sum_{i=1}^N|\partial_i f_m-\partial_i f_n|^2
    \right)
    p_{r,N}^\eta\,dx
    \longrightarrow 0.
\end{equation}
On \(J^2\times K\), Assumption~\ref{ass:lqg_collision} gives
\begin{equation}\label{eq:lqg_density_lower}
    p_{r,N}^\eta(x,y,z)
    \ge
    c_\ast |x-y|^\beta
\end{equation}
for a.e. \((x,y,z)\in J^2\times K\). Therefore
\[
\begin{aligned}
&\int_K
    \|f_m(\cdot,\cdot,z)-f_n(\cdot,\cdot,z)\|_{H^1_w(J^2)}^2\,dz      \\
&\quad
\le
\frac1{c_\ast}
\int_K\int_{J^2}
\left(
    |f_m-f_n|^2
    +
    |\partial_x f_m-\partial_x f_n|^2
    +
    |\partial_y f_m-\partial_y f_n|^2
\right)
p_{r,N}^\eta(x,y,z)\,dx\,dy\,dz .
\end{aligned}
\]
By \eqref{eq:lqg_sector_cauchy}, the right-hand side tends to zero. Hence
\[
    z\mapsto f_m(\cdot,\cdot,z)
\]
is Cauchy in \(L^2(K;H^1_w(J^2))\). Moreover, the same estimate with only the \(L^2_w\)-part shows that
\(f_m\to u_r^{\eta,N}\) in \(L^2(K;L^2_w(J^2))\). Hence the
\(L^2(K;H^1_w(J^2))\)-limit agrees with the slice of the sector representative
\(u_r^{\eta,N}\).
Consequently, for Lebesgue-a.e.~\(z\in K\),
\[
    v_{\eta,z}\in H^1_w(J^2).
\]

We now transfer the admissibility condition \(u\ge1\) on \(O\) to the slices.
Since \(u\ge1\) \(\mu\)-a.e. on \(O\),
\[
    (1-u)_+=0
    \qquad
    \mu\text{-a.e. on }O.
\]
Using the disintegration of \(\mu\), then restricting to the
\(N\)-particle sector, and using \eqref{eq:lqg_tube}, we get
\[
\begin{aligned}
0
&=
\int_O (1-u)_+^2\,d\mu                                                     \\
&\ge
\frac1{N!}
\int_{G_0}\int_K\int_W
\left(
    1-u_r^{\eta,N}(x,y,z)
\right)_+^2
p_{r,N}^\eta(x,y,z)
\,dx\,dy\,dz\,\mu_{I_r^c}(d\eta).
\end{aligned}
\]
Therefore, for \(\mu_{I_r^c}(d\eta)\,dz\)-a.e.
\(
    (\eta,z)\in G_0\times K,
\)
we have
\[
    v_{\eta,z}(x,y)\ge1
\]
for \(p_{r,N}^\eta(x,y,z)\,dx\,dy\)-a.e. \((x,y)\in W\). By
\eqref{eq:lqg_density_lower}, this implies
\[
    v_{\eta,z}(x,y)\ge1
\]
for \(dxdy\)-a.e. \((x,y)\in W\). This proves the claim.
\end{proof}

We resume the proof of Theorem~\ref{thm:lqg_collision}.
Applying Lemma~\ref{lem:weighted_two_particle} and the above Claim to \(v_{\eta,z}\), we obtain
\[
\begin{aligned}
    \int_{J^2}
    \left(
        |v_{\eta,z}|^2
        +
        |\partial_xv_{\eta,z}|^2
        +
        |\partial_yv_{\eta,z}|^2
    \right)
    |x-y|^\beta\,dx\,dy
    \ge
    c_{L,J,\beta}
\end{aligned}
\]
for \(\mu_{I_r^c}(d\eta)\,dz\)-a.e. \((\eta,z)\in G_0\times K\). Combining this with
\eqref{eq:lqg_density_lower}, we get
\begin{equation}\label{eq:lqg_integrated_lower_bound}
\begin{aligned}
&\int_{G_0}\int_K\int_{J^2}
\left(
    |u_r^{\eta,N}(x,y,z)|^2
    +
    |\partial_xu_r^{\eta,N}(x,y,z)|^2
    +
    |\partial_yu_r^{\eta,N}(x,y,z)|^2
\right)
p_{r,N}^\eta(x,y,z)
\,dx\,dy\,dz\,\mu_{I_r^c}(d\eta)                                      \\
&\qquad
\ge
c_\ast c_{L,J,\beta}\,
\mu_{I_r^c}(G_0)|K|.
\end{aligned}
\end{equation}

Next, compare the left-hand side with the local energy and the \(L^2(\mu)\)-norm.
The \(N\)-particle sector contribution gives
\[
\begin{aligned}
    \mathcal E_r(u,u)
    &\ge
    \frac1{2N!}
    \int_{G_0}\int_K\int_{J^2}
    \left(
        |\partial_xu_r^{\eta,N}|^2
        +
        |\partial_yu_r^{\eta,N}|^2
    \right)
    p_{r,N}^\eta
    \,dx\,dy\,dz\,\mu_{I_r^c}(d\eta),
\end{aligned}
\]
and
\[
\begin{aligned}
    \|u\|_{L^2(\mu)}^2
    &\ge
    \frac1{N!}
    \int_{G_0}\int_K\int_{J^2}
    |u_r^{\eta,N}|^2
    p_{r,N}^\eta
    \,dx\,dy\,dz\,\mu_{I_r^c}(d\eta).
\end{aligned}
\]
Therefore, by \eqref{eq:lqg_integrated_lower_bound},
\[
\begin{aligned}
    \mathcal E_r(u,u)+\|u\|_{L^2(\mu)}^2
    &\ge
    \frac1{2N!}
    \int_{G_0}\int_K\int_{J^2}
    \left(
        |u_r^{\eta,N}|^2
        +
        |\partial_xu_r^{\eta,N}|^2
        +
        |\partial_yu_r^{\eta,N}|^2
    \right)
    p_{r,N}^\eta
    \,dx\,dy\,dz\,\mu_{I_r^c}(d\eta)                                   \\
    &\ge
    \frac{c_\ast c_{L,J,\beta}}{2N!}
    \mu_{I_r^c}(G_0)|K|.
\end{aligned}
\]
Set
\[
    c_{\mathrm{lqg}}
    \define
    \frac{c_\ast c_{L,J,\beta}}{2N!}
    \mu_{I_r^c}(G_0)|K|.
\]
Then \(c_{\mathrm{lqg}}>0\). By \eqref{eq:lqg_local_domination},
\[
    \mathcal E(u,u)+\|u\|_{L^2(\mu)}^2
    \ge
    \mathcal E_r(u,u)+\|u\|_{L^2(\mu)}^2
    \ge
    c_{\mathrm{lqg}}.
\]

This lower bound holds for every open \(O\supset\mathscr C_0\) and every
\(u\in\mathcal D\) admissible for \(O\). Hence
\[
    \Cap(O)\ge c_{\mathrm{lqg}}
    \qquad
    \text{for every open }O\supset\mathscr C_0.
\]
Taking the infimum over all such \(O\), we obtain
\[
    \Cap(\mathscr C_0)
    \ge
    c_{\mathrm{lqg}}
    >0.
\]
Since \(\mathscr C_0\subset\mathscr C\), monotonicity gives
\begin{align} \label{e:PC}
    \Cap(\mathscr C)>0.
\end{align}

Suppose now that \((\mathcal E,\mathcal D)\) is quasi-regular, thus admits a properly
associated diffusion~$(\mathbf X_t, \mathbf P_\gamma)$. By~\eqref{e:PC} combined with~Chen--Fukushima~\cite[Theorem~3.1.3]{ChenFukushima2012},
\[
    \mathbf P_\mu[\tau_{\mathscr C}<\infty]>0.
\]
It remains to prove the final assertion under irreducibility. 
Since \(\mu\) is a probability measure and \(1\in\mathcal D\) with
\[
    \mathcal E(1,1)=0,
\]
the closed form \((\mathcal E,\mathcal D)\) is recurrent by~\cite[Theorem~2.1.8]{ChenFukushima2012}. Since it is also assumed
irreducible, the associated symmetric diffusion process is irreducible recurrent.
Thus, by~\cite[Theorem~3.5.6(ii)]{ChenFukushima2012},
we obtain
\[
    \mathbf P_\gamma[\tau_{\mathscr C}<\infty]=1
    \qquad\text{for q.e. }\gamma .
\]
This completes the proof.
\end{proof}

\section{Examples}

\begin{proof}[Proof of Theorem~\ref{DysonThm}]
We split the proof into the two regimes.

\medskip
\noindent{\bf The case \(\beta\ge1\).}
By the correlation formulae and diagonal asymptotics for the \(\mathsf{Sine}_\beta\)
process obtained in~\cite{AssiotisNajnudel}, the second and third factorial
correlation densities \(\rho_2\) and \(\rho_3\) exist and satisfy the local
repulsion bounds required in Assumption~\ref{ass:diag_repulsion}. Hence
Theorem~\ref{thm:noncollision} gives
\[
\Cap_\beta(\mathscr C)=0 .
\]

\medskip
\noindent{\bf The case \(0<\beta<1\).}
We verify Assumption~\ref{ass:lqg_collision}. Since \(\mathsf{Sine}_\beta\) has
intensity one, for every \(r>0\),
\[
\int \theta(I_r)\,\mu(d\theta)=2r.
\]
Choose \(r>1\). Then
\[
\mu(\theta(I_r)\ge2)>0.
\]
Indeed, otherwise \(\theta(I_r)\le1\) \(\mu\)-a.s., which would imply
\[
\int \theta(I_r)\,\mu(d\theta)\le1,
\]
contradicting \(2r>1\). Therefore
\[
\sum_{n=2}^{\infty}\mu(\theta(I_r)=n)
=
\mu(\theta(I_r)\ge2)>0,
\]
and hence there exists a fixed integer \(n\ge2\) such that
\[
\mu(\theta(I_r)=n)>0.
\]

By number rigidity for \(\mathsf{Sine}_\beta\) by~\cite{ChhaibiNajnudel2018, DLR}, there exists a measurable function
\(
N_r:\Conf(I_r^c)\to \mathbb N\cup\{0\}
\)
such that
\[
\mu_{I_r}^{\eta}(\theta(I_r)=N_r(\eta))=1
\qquad
\text{for \(\mu_{I_r^c}\)-a.e. }\eta .
\]
Set
\(
G^{(n)}\define\{\eta\in\Conf(I_r^c):N_r(\eta)=n\}.
\)
Then
\[
\mu(\theta(I_r)=n)
=
\int \mu_{I_r}^{\eta}(\theta(I_r)=n)\,\mu_{I_r^c}(d\eta)
=
\mu_{I_r^c}(G^{(n)}),
\]
and hence
\[
\mu_{I_r^c}(G^{(n)})>0.
\]

By the DLR description of \(\mathsf{Sine}_\beta\) in~\cite{DLR}, the condition~\ref{e:EXD} in Assumption~\ref{ass:lqg_collision} holds for any $r>0$. Thus,  after removing a
\(\mu_{I_r^c}\)-null subset from \(G^{(n)}\), we may assume that for every
\(\eta\in G^{(n)}\),
\[
\mu_{I_r}^{\eta}(\theta(I_r)=n)=1,
\]
and that the conditional law in \(I_r\) admits a symmetric density
\[
q_{r,n}^{\eta}(x_1,\ldots,x_n)
=
\prod_{1\le i<j\le n}|x_i-x_j|^\beta\,
h_{r,n}^{\eta}(x_1,\ldots,x_n),
\]
where \(x\mapsto h_{r,n}^{\eta}(x)\) is  a continuous strictly positive bounded function on
\(I_r^n\) for $\mu_{I_r^c}$-a.e.~$\eta$, and $(x, \eta) \mapsto h_{r,n}^{\eta}(x)$ is jointly measurable; see~\cite[Theorem~1.1, Theorem~2.1, and Lemma~2.3]{DLR}.

On \(G^{(n)}\), the conditional law is supported on the \(n\)-particle sector.
Therefore the unnormalised \(n\)-sector density in Assumption~\ref{ass:lqg_collision}
is simply
\[
p_{r,n}^{\eta}=n!q_{r,n}^{\eta}.
\]
Set
\(
k = n-2.
\)
Choose a bounded open interval \(J\subset I_r\) such that
\[
    \overline J\subset I_r .
\]
If \(k=0\), we use the convention
\(
K=\{\ast\}\) and \(|K|=1.
\)
If \(k\ge1\), choose a compact set
\[
K\subset (I_r\setminus \overline J)^k
\]
of positive \(k\)-dimensional Lebesgue measure such that
\[
z_i\neq z_j \quad (i\neq j)
\]
for every \(z=(z_1,\ldots,z_k)\in K\). We choose \(K\) so that it has positive
distance both from \(J\) and from the diagonals. Equivalently, there exists
\(d_0>0\) such that for all \(x,y\in J\) and all \(z=(z_1,\ldots,z_k)\in K\),
\[
|x-z_j|\ge d_0,\qquad |y-z_j|\ge d_0
\quad (1\le j\le k),
\]
and
\[
|z_a-z_b|\ge d_0
\quad (1\le a<b\le k).
\]

We next extract a positive-measure exterior set on which the factor
\(h_{r,n}^{\eta}\) is uniformly bounded from below on~
\(\overline J^2\times K\). When \(k=0\), interpret this  as
\(\overline J^2\).

Let
\[
\Lambda \define
\begin{cases}
\overline J^2, & k=0,\\
\overline J^2\times K, & k\ge1.
\end{cases}
\]
Since \(h_{r,n}^{\eta}\) is continuous and strictly positive on the compact set
\(\Lambda\), for every \(\eta\in G^{(n)}\),
\[
m(\eta)\define\inf_{u\in\Lambda}h_{r,n}^{\eta}(u)>0.
\]
Taking a countable dense subset \(D\subset \Lambda\), we have
\(
m(\eta)=\inf_{u\in D}h_{r,n}^{\eta}(u).
\)
Combined with the joint-measurability~$(x, \eta) \mapsto h_{r, n}^\eta(x)$, we obtain that \(m\) is measurable. Define
\[
G_\ell\define\{\eta\in G^{(n)}:m(\eta)\ge 1/\ell\},
\qquad \ell\in\mathbb N.
\]
Then
\[
G^{(n)}=\bigcup_{\ell=1}^{\infty}G_\ell.
\]
Since \(\mu_{I_r^c}(G^{(n)})>0\), there exists \(\ell\in\mathbb N\) such that
\(
\mu_{I_r^c}(G_\ell)>0.
\)
Set
\[
G = G_\ell,
\qquad
c_h = 1/\ell.
\]
Then, for every \(\eta\in G\),
\[
h_{r,n}^{\eta}(u)\ge c_h
\qquad
(u\in\Lambda).
\]

We now prove the lower bound required in Assumption~\ref{ass:lqg_collision}.
If \(k=0\), then for a.e. \((x,y)\in J^2\),
\[
p_{r,2}^{\eta}(x,y)
=
2!q_{r,2}^{\eta}(x,y)
=
2|x-y|^\beta h_{r,2}^{\eta}(x,y)
\ge
2c_h |x-y|^\beta .
\]
Thus Assumption~\ref{ass:lqg_collision} holds with \(c_*=c_h\).

If \(k\ge1\), then for a.e. \((x,y,z)\in J^2\times K\),
\[
p_{r,n}^{\eta}(x,y,z)
=
n!q_{r,n}^{\eta}(x,y,z)
=
n!\prod_{1\le i<j\le n}|u_i-u_j|^\beta\,
h_{r,n}^{\eta}(u),
\]
where
\(
u=(u_1,\ldots,u_n)=(x,y,z_1,\ldots,z_k).
\)
By the separation property of \(K\), all factors in the Vandermonde product
except \(|x-y|^\beta\) are bounded below by a positive constant. More precisely,
\[
\prod_{1\le i<j\le n}|u_i-u_j|^\beta
\ge
d_0^{\beta\left(2k+\binom{k}{2}\right)}|x-y|^\beta .
\]
Therefore
\[
p_{r,n}^{\eta}(x,y,z)
\ge
n!c_h\,d_0^{\beta\left(2k+\binom{k}{2}\right)}|x-y|^\beta .
\]
Hence Assumption~\ref{ass:lqg_collision} holds with
\[
c_*\define n!c_h\,d_0^{\beta\left(2k+\binom{k}{2}\right)}.
\]

Thus Assumption~\ref{ass:lqg_collision} is verified for the fixed integer \(k=n-2\),
the positive-measure exterior set \(G\), and the compact set \(K\). By
Theorem~\ref{thm:lqg_collision},
\[
\Cap_\beta(\mathscr C)>0.
\]

\paragraph{The closability and the quasi-regularity}
The closability can be seen, e.g., by Lemma~\ref{l:CLE} with the observation that \ref{e:EXD}~of~Assumption~\ref{ass:lqg_collision} holds thanks to the DLR equation in~\cite{DLR} as seen above.
The quasi-regularity follows by applying~\cite[Theorem 1]{Osa96}. To do so, it suffices to check 
$$\sum_{n\ge 1} n \mu(\Theta_r^n)<\infty, \qquad \rho_n \in L^\infty(I_r^n, dx) \quad \forall r>0,$$
where $\rho_n(x_1, \ldots, x_n)$ is the $n$-th factorial correlation density. The first estimate about the series is immediate by noting that $\Sine_\beta$ has the unit intensity $\rho_1$, which implies 
$$\sum_{n\ge 1} n \mu(\Theta_r^n) = \int_\Theta \theta(I_r) \mu(d\theta)= \int_{-r}^r \rho_1(x)dx<\infty.$$
The second part is a direct corollary of~the correlation function estimates obtained in~\cite{AssiotisNajnudel}.

Combining the above,  the theorem has been proved.
\end{proof}

Recall the definition of a determinantal measure $\mu=\mu_{\mathsf{K}}$ associated to a Hermitian correlation kernel $\mathsf{K}:\mathbb{R}\times \mathbb{R}\to \mathbb{C}$ from \cite{BorodinDet,Soshnikov00,JohanssonDet}.

\begin{prop}\label{prop:DetermintalCap}
Let $\mu=\mu_\mathsf{K}$ be a determinantal measure and suppose the kernel $\mathsf{K}$ is jointly locally Lipschitz continuous. Assume in addition that the pre-Dirichlet form~$(\cE, \cD_\infty)$ is closable. Then, $\Cap(\mathscr{C})=0$.
\end{prop}

\begin{proof}
Since $\mu_\mathsf{K}$ is determinantal we have that $\rho_2(x_1,x_2)$ and $\rho_3(x_1,x_2,x_3)$ are given as determinants of $[\mathsf{K}(x_i,x_j)]_{i,j=1}^M$ with $M=2,3$. By virtue of the joint local Lipschitz continuity of $\mathsf{K}$, the bounds \eqref{eq:rho2-repulsion}, \eqref{eq:rho3-repulsion} readily follow from elementary manipulations. Indeed, by joint local Lipschitz continuity of \(K\), the determinant defining
\(\rho_2\) vanishes at least linearly when two variables coalesce, and the
determinant defining \(\rho_3\) vanishes at least linearly near the triple
diagonal. Hence the bounds in Assumption~\ref{ass:diag_repulsion} hold with \(\alpha=1\).
\end{proof}

\bibliographystyle{acm}
\bibliography{References}

\bigskip 

\noindent{\sc School of Mathematics, University of Edinburgh, James Clerk Maxwell Building, Peter Guthrie Tait Rd, Edinburgh EH9 3FD, U.K.}\newline
\href{mailto:theo.assiotis@ed.ac.uk}{\small theo.assiotis@ed.ac.uk}

\bigskip

\noindent{\sc Department of Mathematical Science, Durham University, DH13LE, Durham, United Kingdom}\newline
\href{mailto:kohei.suzuki@durham.ac.uk}{\small kohei.suzuki@durham.ac.uk}

\end{document}